\documentclass[12pt,leqno]{article}
\usepackage{amsfonts}
\usepackage{amsmath, amsthm, amsfonts, amssymb, color}
\usepackage{mathrsfs}
\newtheorem{thm}{Theorem}[section]

\newtheorem{lem}[thm]{Lemma}

\theoremstyle{definition}

\newcommand{\scr}[1]{\mathscr #1}
\definecolor{wco}{rgb}{0.5,0.2,0.3}

\numberwithin{equation}{section} \theoremstyle{remark}
\newtheorem{rem}{Remark}[section]
\renewcommand{\bar}{\overline}
\renewcommand{\tilde}{\widetilde}

\newcommand{\ua}{\uparrow}

\title{{\bf Hypercontractivity and Its Applications for  Functional SDEs of Neutral Type \thanks{supported in part by NSFC (No.11401592).}}
}
\author{
{\bf  Jianhai Bao$^{a}$,   Chenggui Yuan$^{b}$}\\
\footnotesize{$^{a}$School of Mathematics and Statistics, Central
South
University, Changsha 410083, China}\\
 \footnotesize{$^{b}$Department of Mathematics,
Swansea University, Singleton Park, SA2 8PP, UK}\\
\footnotesize{jianhaibao13@gmail.com,   C.Yuan@swansea.ac.uk}}
\begin{document}
\def\R{\mathbb R}  \def\ff{\frac} \def\ss{\sqrt} \def\B{\mathbf
B}
\def\N{\mathbb N} \def\kk{\kappa} \def\m{{\bf m}}
\def\dd{\delta} \def\DD{\Dd} \def\vv{\varepsilon} \def\rr{\rho}
\def\<{\langle} \def\>{\rangle} \def\GG{\Gamma} \def\gg{\gamma}
  \def\nn{\nabla} \def\pp{\partial} \def\EE{\scr E}
\def\d{\text{\rm{d}}} \def\bb{\beta} \def\aa{\alpha} \def\D{\scr D}
  \def\si{\sigma} \def\ess{\text{\rm{ess}}}
\def\beg{\begin} \def\beq{\begin{equation}}  \def\F{\scr F}
\def\Ric{\text{\rm{Ric}}} \def\Hess{\text{\rm{Hess}}}
\def\e{\text{\rm{e}}} \def\ua{\underline a} \def\OO{\Omega}  \def\oo{\omega}
 \def\tt{\tilde} \def\Ric{\text{\rm{Ric}}}
\def\cut{\text{\rm{cut}}} \def\P{\mathbb P} \def\ifn{I_n(f^{\bigotimes n})}
\def\C{\scr C}      \def\aaa{\mathbf{r}}     \def\r{r}
\def\gap{\text{\rm{gap}}} \def\prr{\pi_{{\bf m},\varrho}}  \def\r{\mathbf r}
\def\Z{\mathbb Z} \def\vrr{\varrho} \def\l{\lambda}
\def\L{\scr L}\def\Tt{\tt} \def\TT{\tt}\def\II{\mathbb I}
\def\i{{\rm in}}\def\Sect{{\rm Sect}}\def\E{\mathbb E} \def\H{\mathbb H}
\def\M{\scr M}\def\Q{\mathbb Q} \def\texto{\text{o}} \def\LL{\Lambda}
\def\Rank{{\rm Rank}} \def\B{\scr B} \def\i{{\rm i}} \def\HR{\hat{\R}^d}
\def\to{\rightarrow}\def\l{\ell}\def\ll{\lambda}
\def\8{\infty}\def\ee{\epsilon}

\def\8{\infty}\def\ee{\epsilon} \def\Y{\mathbb{Y}} \def\lf{\lfloor}
\def\rf{\rfloor}\def\3{\triangle}\def\H{\mathbb{H}}\def\S{\mathbb{S}}\def\1{\lesssim}
\def\va{\varphi}

\maketitle

\begin{abstract}
In this paper, we discuss  hypercontractivity for the Markov
semigroup $P_t$ which is generated by  segment processes associated
with a range of functional SDEs of neutral type. As applications, we
also reveal that the semigroup $P_t$ converges exponentially to its
unique invariant probability measure $\mu$ in entropy, $L^2 (\mu)$
and   $\|\cdot\|_{\mbox{var}}$, respectively.
 \

 \end{abstract}

\section{Introduction}
According to \cite{Gross} by Gross, the Markov semigroup $P_t$ is
called hypercontractivity with respect to the invariant probability
measure $\mu$ if $\|P_t\|_{2\rightarrow4}\le1$ for large $t>0,$
where $\|\cdot\|_{2\to4}$ stands for the operator norm from
$L^2(\mu)$ to $L^4(\mu)$. The hypercontractivity of Markov
semigroups has been extensively studied for various models (see,
e.g., \cite{Bakry,BWY, Davies, Gross, Wbook1, Wbook2, W13,Wang14}
and references therein). In the light of \cite{Gross}, the
log-Sobolev inequality implies the hypercontractivity. However, the
approach adopted in \cite{Gross} no longer works for functional SDEs
since the log-Sobolev inequality for the associated Dirichlet form
is invalid. In \cite{BWY}, utilizing the Harnack inequality with
power initiated in \cite{W97}, the authors investigated the
hypercontractivity and its applications for a range of
non-degenerate functional SDEs.  Wang \cite{Wang14} developed a
general framework on how to establish the hypercontractivity for
Markov semigroups (see \cite[Theorem 2.1]{Wang14}). Meanwhile, he
applied successfully the theory  to finite/infinite dimensional
stochastic Hamiltonian systems.

\smallskip
In this paper, as a continuation of our work \cite{BWY}, we are
still interested in the hypercontractivity and its applications,
however, for a kind of functional SDEs of {\it neutral type}. For a
functional SDE of neutral type, we mean an SDE  which not only
depends on the past and the present values but also involves
derivatives with delays (see, e.g., \cite[Chapter 6]{M08}). Such
equation has also been utilized to model some evolution phenomena
arising in, e.g., physics, biology and engineering. See, e.g.,
Kolmanovskii-Nosov \cite{KN} concerning the theory in
aeroelasticity,
 Mao \cite{M08} with regard to the collision problem in electrodynamics, and Slemrod
\cite{S71} for the oscillatory systems, to name a few. For large
deviation of functional SDEs of neutral type, we refer to Bao-Yuan
\cite{BY}.

\smallskip
We introduce some notation. Let $(\R^n,\<\cdot,\cdot\>,|\cdot|)$ be
an $n$-dimensional Euclidean space, and $\{W(t)\}_{t\ge0}$   an
$n$-dimensional Brownian motion defined on a complete filtered
 probability space $(\OO,\F,\{\F_t\}_{t\ge0}, \P)$.
 For  a closed interval
 $I\subset\R$,
 $C(I;\R^n)$ denotes the collection of all continuous functions
 $f:I\mapsto\R^n$. For a fixed constant $r_0>0$, let $\C:=C([-r_0,0];\R^n)$
 endowed with the uniform norm
 $\|f\|_\8:=\sup_{-r_0\le\theta\le0}|f(\theta)|$ for $f\in\C.$ For $X(\cdot)\in C([-r_0,\8);\R^n)$ and $t\ge0$,
define the segment process $X_t\in\C$ by $X_t(\theta):=X(t+\theta)$,
$\theta\in[-r_0,0]$. $\R^n\otimes\R^n$ means  the family of all
$n\times n$ matrices. $\scr P(\C)$ stands for the set of all
probability measures on $\C$, and $\|\cdot\|_{\mbox{var}}$ denotes
the total variation norm.

\smallskip

In the paper, we  focus on a   functional SDE of neutral type in the
framework
\begin{equation}\label{eq0}
\d\{X(t)+LX_t\}=\{Z(X(t))+b(X_t)\}\d t+\si\d
W(t),~~~t>0,~~X_0=\xi\in\C,
\end{equation}
where, for any $\phi\in\C$,
\begin{equation}\label{r1}
L\phi:=\kk\int_{-r_0}^0\phi(\theta)\d\theta, \mbox{  with } \kk\in(0,1),
\end{equation}
$Z:\R^n\mapsto\R^n$ and $b:\C\mapsto\mathbb{R}^n$ are progressively
measurable, and $\si\in\R^n\otimes\R^n$ is an invertible matrix.

\smallskip
Throughout the paper, for any $x,y\in\R^n$ and $ \xi,\eta\in\C$, we
assume that
\begin{enumerate}
\item[(H1)] $Z$ and $b$ are Lipschitzian with Lipschitz
constants $L_1>0$ and $L_2>0$, respectively, i.e., $|Z(x)-Z(y)|\le
L_1|x-y|$ and $ |b(\xi)-b(\eta)|\le L_2\|\xi-\eta\|_\8;$

\item[(H2)] There exist constants $\ll_1>\ll_2>0$ such that
\begin{equation*}
2\<Z(\xi(0))-Z(\eta(0))+b(\xi)-b(\eta),
\xi(0)-\eta(0)+L(\xi-\eta)\>\le \ll_2
\|\xi-\eta\|_\infty^2-\ll_1|\xi(0)-\eta(0)|^2.
 \end{equation*}
\end{enumerate}

On the basis of  (H1) and $\kk\in(0,1),$ \eqref{eq0} admits a unique
strong solution $\{X(t)\}_{t\ge-r_0}$ (see, e.g., \cite[Theorem 2.2,
p204]{M08}). The dissipative-type condition (H2) is imposed to
reveal the long-time behavior of segment process $\{X_t\}_{t\ge0}$.
For more details, please refer to Lemmas
\ref{exponential}-\ref{inva} below. Moreover, for illustrative
examples such that (H2) holds, we would like to refer to, e.g.,
\cite[Example 5.2]{MSY}. On occasions, to emphasize the initial
datum $X_0=\xi\in\C$, we write $\{X(t;\xi)\}_{t\ge-r_0}$ instead of
$\{X(t)\}_{t\ge-r_0}$, and $\{X_t(\xi)\}_{t\ge0}$ in lieu of
$\{X_t\}_{t\ge0}$, respectively.

\smallskip
 As we present in the beginning of the
second paragraph, in this paper, we intend to investigate the
hypercontractivity and its applications for the Markov semigroup
\begin{equation}\label{eq5}
P_tf(\xi):=\E f(X_t(\xi)),~f\in\B_b(\C),~\xi\in\C.
\end{equation}

\smallskip

Our main result in this paper is stated as below.

\begin{thm}\label{T1.1}
{\rm  Let (H1)-(H2) hold and assume  $\kk\in(0,1)$, and suppose
further that
\begin{equation}\label{07}
1-\kk r_0^2 \e^{\rr r_0}>0~~\mbox{ and
}~~\ll:=\rr-\ff{(\kk r_0^2\ll_1+\ll_2)\e^{\rr r_0}}{(1-\kk)(1-\kk r_0^2\e^{\rr
r_0})}>0.
\end{equation}
 with $\rr=\ll_1/(1+\kk)$. Then, the following assertions hold.
\beg{enumerate}
\item[$(1)$] $P_t$ has a unique invariant
probability measure $\mu$.
\item[$(2)$] $P_t$ is hypercontractive.
\item[$(3)$] $P_t$ is compact on $L^2(\mu)$ for large enough $t>0$,
and there exist $c,\aa>0$ such that
\begin{equation*}
\mu((P_tf)\log P_t f)\le c\e^{-\aa t}\mu(f\log f)
\end{equation*}

\item[$(4)$] There exists a constant $C>0$ such that
$$\|P_t-\mu\|_2^2:= \sup_{\mu(f^2)\le 1}\mu\big((P_t f-\mu(f))^2\big)\le C\e^{-\ll t},\ \ t\ge 0.$$
\item[$(5)$] There exist two constants $t_0, C>0$ such that
$$\|P_t^\xi-P_t^\eta\|_{\mbox{var}}^2\le  C\|\xi-\eta\|_\infty^2 \e^{-\ll t},\ \ t\ge t_0,$$ where   $P_t^\xi$ stands for  the law of $X_t^\xi$ for   $(t,\xi)\in
 [0,\infty)\times\C$.
\end{enumerate}
}
\end{thm}

Compared with \cite{BWY,WY,Wang14}, this paper contains the
following new points: (1) Theorem \ref{T1.1} works for functional
SDEs of {\it neutral type} and covers \cite[Theorem  1.1]{BWY}
whenever $\kk=0$; (2) The argument of Lemma \ref{Harnack} gives some
new ideas on how to establish the dimensional-free Harnack
inequality for functional SDEs of neutral type.  The remainder of
this paper is organized as follows. In Section \ref{sec2},  we
investigate Gauss-type concentration property (see Lemma
\ref{exponential}), existence and uniqueness of invariant measure
(see Lemma \ref{inva}), and the Harnack inequality (see Lemma
\ref{Harnack}), and  devote to completing the proof of Theorem
\ref{T1.1}.

\section{Proof of Theorem \ref{T1.1}}\label{sec2}

\cite[Theorem 2.1]{Wang14}, which is stated as Lemma \ref{general}
below for presentation convenience, establishes a general result on
the hypercontractivity of Markov semigroups. For the present
situation, the key point in the proof of Theorem \ref{T1.1} is to
realize (i)-(iii) in \cite[Theorem 2.1]{Wang14}, one by one.
\smallskip

 Let $(E,\B,\mu)$ be a probability space, and
$P_t$ a Markov semigroup on $\B_b(E)$ such that $\mu$ is
$P_t$-invariant. Recall that a process $(X_t,Y_t)$ on $E\times E$ is
called a coupling associated with the semigroup $P_t$ if
\begin{equation*}
P_tf(\xi)=\E(f(X_t)|\xi),~~~P_tf(\eta)=\E(f(Y_t)|\eta),~~f\in\B_b(E),~t\ge0.
\end{equation*}

\begin{lem}\label{general}
{\rm Assume that the following three conditions hold for some
measurable functions $\rr:E\times E\mapsto(0,\8)$ and
$\phi:[0,\8)\mapsto(0,\8)$ with $\lim_{t\to\8}\phi(t)=0:$
\begin{enumerate}
\item[(i)] There exist constants $t_0,c_0>0$ such that
\begin{equation*}
(P_{t_0}f(\xi))^2\le
(P_{t_0}f^2(\eta))\e^{c_0\rr(\xi,\eta)^2},~~~f\in\B_b(E),~\xi,\eta\in
E;
\end{equation*}

\item[(ii)] For any $\xi,\eta\in E\times E$, there exists a coupling
$X_t,Y_t$ associated with $P_t$ such that
\begin{equation*}
\rr(X_t,Y_t)\le\phi(t)\rr(\xi,\eta),~~~t\ge0;
\end{equation*}

\item[(iii)] There exists $\vv>0$ such that
$(\mu\times\mu)(\e^{\vv\rr^2})<\8$.
\end{enumerate}
Then $\mu$ is the unique invariant probability measure of $P_t$,
$P_t$ is hypercontractive and compact in $L^2(\mu)$ for large $t>0.$

}
\end{lem}

\smallskip

Hereinafter, we first investigate the following exponential-type
estimate, which plays a crucial role in discussing the
hypercontractivity.

\begin{lem}\label{exponential}
{\rm Under the assumptions of Theorem \ref{T1.1}, there exist
$\vv,c>0$ such that
\begin{equation}\label{09}
\E\, \e^{\vv \|X_t^\xi\|_\infty^2} \le \e^{c(1+\|\xi\|_\infty^2)},\
\ t\ge 0,~ \xi\in \C.\end{equation}}
\end{lem}

\begin{proof}
For $\eta(\theta)\equiv0,\theta\in[-r_0,0]$, by (H2) and \eqref{07},
we deduce that
\begin{equation}\label{06}
\begin{split}
&2\<\xi(0)+L\xi, Z(\xi(0))+b(\xi)\>\\
&=2\<\xi(0)-\eta(0)+L(\xi-\eta),
Z(\xi(0))-Z(\eta(0))+b(\xi)-b(\eta)+Z(\eta(0))+b(\eta)\>\\
&\le c_0-\ll_1^\prime|\xi(0)|^2+\ll_2^\prime\|\xi\|_\8^2
\end{split}
\end{equation}
for some constants $c_0>0$ and $\ll_1^\prime,\ll_2^\prime>0$ such
that
\begin{equation}\label{08}
1-\kk r_0^2\e^{\rr^\prime
r_0}>0,~~~~\ll^\prime:=\rr^\prime-\ff{(\kk r_0^2\ll_1^\prime+\ll_2^\prime)\e^{\rr^\prime
r_0}}{(1-\kk)(1-\kk r_0^2\e^{\rr^\prime r_0})}>0,
\end{equation}
where $\rr^\prime=\ll_1^\prime/(1+\kk)$. For simplicity, let $
\GG(t):=X(t)+LX_t. $ By It\^o's formula,  it follows from \eqref{06}
that
\begin{equation}\label{eq6}
\d\{\e^{\rr^\prime t}|\GG(t)|^2\} \le\e^{\rr^\prime
t}\{\rr^\prime|\GG(t)|^2+c_0-\ll_1^\prime|X(t)|^2+\ll_2^\prime\|X_t\|^2_\8\}\d
t+\d M(t),
\end{equation}
where $\d M(t):=2\e^{\rr^\prime t}\<\GG(t),\si\d W(t)\>$. Recall the
following elementary inequality:
\begin{equation}\label{q2}
(a+b)^2\le (1+\dd)a^2+(1+\dd^{-1})b^2,~a,b\in\R,~\dd>0.
\end{equation}
So one has
\begin{equation*}
|\GG(t)|^2\le(1+\kk)|X(t)|^2+(1+\kk^{-1})|LX_t|^2,
\end{equation*}
which, together with $|L\phi|\le \kk\|\phi\|_\8 $ for $\phi\in\C$
and $\kk\in(0,1),$ leads to
\begin{equation*}
-|X(t)|^2\le -\ff{1}{1+\kk}|\GG(t)|^2+\kk r_0^2\|X_t\|^2_\8.
\end{equation*}
Substituting this into \eqref{eq6} gives that
\begin{equation}\label{04}
\e^{\rr^\prime
t}|\GG(t)|^2\le|\GG(0)|^2+\ff{c_0}{\rr^\prime}\e^{\rr^\prime
t}+N(t)+(\kk r_0^2\ll_1^\prime+\ll_2^\prime)\int_0^t\e^{\rr^\prime
s}\|X_s\|^2_\8\d s,
\end{equation}
where $N(t):=\sup_{0\le s\le t}M(s)$.  Choosing
$\dd=\ff{\kk}{1-\kk}$ in \eqref{q2} and noting $|L\phi|\le
\kk r_0\|\phi\|_\8$ for  $\phi\in\C$ and $\kk\in(0,1)$ into
consideration, we derive that
\begin{equation*}
\begin{split}
\e^{\rr^\prime s}|X(s)|^2 &\le\ff{1}{1-\kk}\e^{\rr^\prime
s}|\GG(s)|^2+\kk r_0^2\e^{\rr^\prime
r_0}\sup_{-r_0\le\theta\le0}(\e^{\rr^\prime(s+\theta)}|X(s+\theta)|^2)\\
&\le\ff{1}{1-\kk}\e^{\rr^\prime s}|\GG(s)|^2+\kk r_0^2\e^{\rr^\prime
r_0}\sup_{s-r_0\le u\le s}(\e^{\rr^\prime u}|X(u)|^2).
\end{split}
\end{equation*}
We then have
\begin{equation*}
\sup_{0\le s\le t}(\e^{\rr^\prime s}|X(s)|^2)
 \le\kk r_0^2\e^{\rr^\prime r_0}\|\xi\|^2_\8+\ff{1}{1-\kk}\sup_{0\le
s\le t}(\e^{\rr^\prime s}|\GG(s)|^2)+\kk r_0^2 \e^{\rr^\prime
r_0}\sup_{0\le s\le t}(\e^{\rr^\prime s}|X(s)|^2)
\end{equation*}
so  that, due to \eqref{08},
\begin{equation}\label{01}
\sup_{0\le s\le t}(\e^{\rr^\prime s}|X(s)|^2)
\le\ff{\kk r_0^2\e^{\rr^\prime r_0}}{1-\kk r_0^2\e^{\rr^\prime
r_0}}\|\xi\|^2_\8+\ll^{\prime\prime}\sup_{0\le s\le
t}(\e^{\rr^\prime s}|\GG(s)|^2),
\end{equation}
in which $\ll^{\prime\prime}:=(1-\kk)^{-1}(1-\kk r_0^2\e^{\rr^\prime
r_0})^{-1}$. Moreover, it is easy to see that
\begin{equation}\label{02}
\e^{\rr^\prime t}\|X_t\|^2_\8\le\e^{\rr^\prime r_0}\sup_{t-r_0\le
s\le t}(\e^{\rr^\prime s}|X(s)|^2).
\end{equation}
Thus, from \eqref{01} and \eqref{02}, we infer that
\begin{equation}\label{03}
\e^{\rr^\prime t}\|X_t\|^2_\8 \le
c_1\|\xi\|^2_\8+\ll^{\prime\prime}\e^{\rr^\prime r_0}\sup_{0\le s\le
t}(\e^{\rr^\prime s}|\GG(s)|^2)
\end{equation}
for some $c_1>0$. Taking \eqref{04} and \eqref{03} into account
yields that
\begin{equation}\label{q3}
\e^{\rr^\prime t}\|X_t\|^2_\8 \le
c_1\|\xi\|^2_\8+c_2(N(t)+\e^{\rr^\prime
t})+\gg^\prime\int_0^t\e^{\rr^\prime s}\|X_s\|^2_\8\d s
\end{equation}
for some $c_2>0$, where $\gamma^\prime:=\ll^{\prime\prime}\e^{\rr'
r_0}(\kk\ll_1^\prime+\ll_2^\prime)$. By Gronwall's inequality and
\eqref{08}, one has
\begin{equation*}
\|X_t\|^2_\8 \le c_3(1+\|\xi\|^2_\8)+c_3\e^{-\rr^\prime
t}N(t)+c_3\int_0^t \e^{-\rr' s-\ll^\prime(t -s)}N(s)\d s
\end{equation*}
for some constant $c_3>0,$  where $\ll^\prime>0$ is defined in
\eqref{08}. Hence, for any $\vv>0$, we have
\begin{equation*}
\begin{split}
\E\e^{\vv\|X_t\|^2_\8}\le \e^{c_3(1+\|\xi\|^2_\8)}\ss{I_1\times
I_2},
\end{split}
\end{equation*}
in which  \beg{equation*}\beg{split} &I_1:= \E \exp\bigg[ 2 c_3\vv
\int_0^{t } \e^{-\rr' s-\ll^\prime(t -s)}N(s)\d s\bigg] ~~~\mbox{
and }~~~ I_2:= \E \exp\big[ 2 c_3\vv\e^{-\rr't}N(t
)\big].\end{split}\end{equation*} Next, on following arguments of
\cite[(2.4) and (2.5)]{BWY}, there exists $\vv>0$ and $c_4>0$ such
that
$$\e^{\ll^\prime t} \E \e^{\vv\|X_t \|_\infty^2}\le \e^{c_4(1+\|\xi\|_\infty^2)+\ll^\prime t}+ \ff {\ll^\prime} 2 \int_{-r_0}^t
 \e^{\ll^\prime s}\E \e^{\vv
\|X_{s }\|_\infty^2}\d s.$$ Consequently, an application of the
Gronwall inequality gives that
\begin{equation*}
\begin{split}
 \E \e^{\vv\|X_t \|_\infty^2}\le
\e^{c_4(1+\|\xi\|_\infty^2)}+\ff {\ll^\prime}
2\int_{-r_0}^t\e^{c_4(1+\|\xi\|_\infty^2)}\e^{-\ff {\ll^\prime
(t-s)} 2}\d s,
\end{split}
\end{equation*}
and the desired assertion \eqref{09} follows immediately due to
$\ll^\prime>0.$

\end{proof}

\begin{rem}
{\rm In Lemma \ref{exponential}, if $\kk=0$, then the first
condition in \eqref{07} is definitely true, and the second one
reduces to $\ll_1>\ll_2\e^{\ll_2r_0}$, which is imposed in
\cite[Lemma 2.1]{BWY} to show the Gauss-type concentration property
of the unique invariant probability measure for a range of
non-degenerate functional SDEs. }
\end{rem}

\begin{rem}
{\rm By an  inspection of argument of Lemma \ref{exponential},
we observe that Lemma \ref{exponential} still holds for a general
neutral term $G:\C\mapsto\R^n$ provided that
$|G(\xi)|\le\kk\|\xi\|_\8$ for   $\xi\in\C$ and $\kk\in(0,1)$.
However, in Lemma \ref{exponential},  we only consider the linear
case $G(\xi)=L\xi,\xi\in\C,$  this is for the consistency with the
Harnack inequality to be established in Lemma \ref{Harnack} below. }
\end{rem}

\smallskip
The lemma below states that segment processes with different initial
datum close to each other when the time tends to infinity.

\begin{lem}\label{long}
{\rm Under the assumptions of Theorem \ref{T1.1}, there exists a constant $c>0$
such that
\begin{equation}\label{w1}
\|X_t(\xi)-X_t(\eta)\|_\8^2\le c\e^{-\ll t}\| \xi-\eta\|^2_\8,
\end{equation}
where $\ll>0$ is given in  \eqref{07}.

}
\end{lem}

\begin{proof}
For simplicity, set
\begin{equation*}
\Phi(t):=X(t;\xi)- X(t;\eta)+L(X_t(\xi)-  X_t(\eta)).
\end{equation*}
 By the chain rule, for $\rr=\ll_1/(1+\kk)$, it follows that
\begin{equation*}
\d(\e^{\rr t}|\Phi(t)|^2) \le\e^{\rr
t}\{\rr|\Phi(t)|^2-\ll_1|X(t;\xi)- X(t;\eta)|^2+\ll_2\|X_t(\xi)-
X_t(\eta)\|_\8^2\}\d t.
\end{equation*}
Thus, carrying out a similar   argument to derive \eqref{q3}, we can
deduce that
\begin{equation*}
\e^{\rr t}\|X_t(\xi)- X_t(\eta)\|^2_\8 \le
C_1\|\xi-\eta\|^2_\8+(\rr-\ll)\int_0^t\e^{\rr s}\|X_s(\xi)-
X_s(\eta)\|^2_\8\d s
\end{equation*}
for some constant $C_1>0$, where $\ll>0$ is defined as in
\eqref{07}. Next, by virtue of   Gronwall's inequality, one finds
\begin{equation*}
\|X_t(\xi)-X_t(\eta)\|^2_\8 \le C_1\e^{-\ll t}\| \xi-\eta\|^2_\8.
\end{equation*}
The proof is, therefore, complete.
\end{proof}

\begin{lem}\label{inva}
{\rm Let the conditions of Theorem \ref{T1.1} hold. Then, the Markov
semigroup $P_t$, defined in \eqref{eq5},   admits a unique invariant
probability measure $\mu\in \scr P(\C)$.

}
\end{lem}

\begin{proof} The method of the proof is similar to that of Lemma 2.4 in \cite{BWY}, for the convenience of the reader, we give a sketch of the proof.  
    Let $W$ be the $L^1$-Wasserstein distance induced by the distance 
$\rr(\xi,\eta):= 1\land \|\xi-\eta\|_\infty$; that is
$$W(\mu_1,\mu_2):= \inf_{\pi\in \C(\mu_1,\mu_2)} \pi(\rr),\ \ \mu_1,\mu_2\in \scr P(\C),$$
where 
 $\C(\mu_1,\mu_2)$ is the set of all couplings of $\mu_1$ and
$\mu_2.$  It is well
known that $\scr P(\C)$ is a complete metric space with respect to
the distance $W$ (see, e.g., \cite[Lemma 5.3 and Lemma 5.4]{Chen}),
and the convergence in $W$ is equivalent to the weak convergence
whenever $\rr$ is bounded (see, e.g., \cite[Theorem 5.6]{Chen}).  Let  $\P_t^\xi$ be the law of $X_t(\xi)$. So,
to show existence of an invariant measure, it is sufficient to claim
that $\P_t^\xi$ is a $W$-Cauchy sequence, i.e.,
\begin{equation}\label{q4}
\lim_{t_1,t_2\to\8} W(\P_{t_1}^\xi,\P_{t_2}^\xi)=0.
\end{equation}

\smallskip

For any $t_2>t_1>0$, consider the following neutral SDEs of neutral
type
\begin{equation*}
\d\{X(t)+LX_t\}=\{Z(X(t))+b(X_t)\}\d t+\si\d
W(t),~t\in[0,t_2],~~X_0=\xi,
\end{equation*}
 and
\begin{equation*}
 \d\{\bar X(t)+L\bar
X_t\}=\{Z(\bar X(t))+b(\bar X_t)\}\d t+\si\d
W(t),~t\in[t_2-t_1,t_2],~~\bar X_{t_2-t_1}=\xi.
\end{equation*}
 Then the laws
of $X_{t_2}(\xi)$ and $\bar X_{t_2}(\xi)$ are $\P_{t_2}^\xi$ and
$\P_{t_1}^\xi$, respectively. Also,  following an argument of
\eqref{q3}, we can obtain that
\begin{equation*}
\e^{\rr t}\E\|X_t-\bar X_t\|^2_\8 \le
C_2\E\|X_{t_2-t_1}-\xi\|^2_\8+(\rr-\ll)\int_{t_2-t_1}^t\e^{\rr
s}\E\|X_s-\bar X_s\|^2_\8\d s
\end{equation*}
for some constant $C_2>0$, in which  $\ll>0$ is defined as in
\eqref{07}. According to the Gronwall inequality, one finds
\begin{equation*}
\E\|X_t-\bar X_t\|^2_\8 \le C_2\e^{-\ll
(t-t_2+t_1)}\E\|X_{t_2-t_1}-\xi\|^2_\8.
\end{equation*}
This, in addition to \eqref{09}, gives that
\begin{equation*}
\E\|X_{t_2}-\bar X_{t_1}\|^2_\8 \le C_2\e^{-\ll t_1},
\end{equation*}
which further implies
\begin{equation*}
W(\P_{t_1}^\xi,\P_{t_2}^\xi)\le \E\|X_{t_2}-\bar X_{t_1}\|_\8\le
\ss{C_2}\e^{-\ff{\ll t_1}{2}}.
\end{equation*}
Then, \eqref{q4} holds by taking $t_1\to\8.$ So, there exists
$\mu^\xi\in \scr P(\C)$ such that
\begin{equation}\label{w2}
\lim_{t\to\8}W(\P_t^\xi,\mu^\xi)=0.
\end{equation}
Thus, the desired assertion follows provided that we can show that
$\mu^\xi$ is independent of $\xi\in\C$. To this end, note that
\begin{equation*}
W(\mu^\xi,\mu^\eta)\le
W(\P_t^\xi,\mu^\xi)+W(\P_t^\eta,\mu^\eta)+W(\P_t^\xi,\P_t^\eta).
\end{equation*}
Taking $t\to\8$ and using \eqref{w1} and \eqref{w2}, we conclude
that  $\mu^\xi\equiv\mu^\eta$ for any $\xi,\eta\in \C$. Hence, we
conclude that  $\mu^\xi$ is independent of $\xi\in\C$.
\end{proof}

The condition (i) in Lemma \ref{general} is concerned with the
dimension-free Harnack inequality which is initiated in \cite{W97}.
To the best of our knowledge, coupling by change of measure (see,
e.g., the monograph \cite{W13}) and Malliavin calculus (see, e.g.,
\cite{BWY2}) are two popular approaches to establish the Harnack
inequality, which has considerable applications in contractivity
properties, functional inequalities, short-time behaviors of
infinite-dimensional diffusions, as well as  heat kernel estimates
(see, e.g., Wang \cite{W13,W07}). For the Harnack inequality of
stochastic partial differential equations (SPDEs), we refer to the
monograph \cite{W13}. For the Harnack inequality of SDEs with
memory, we would like to refer to \cite{ES} for functional SDEs with
additive noises, \cite{WY} for functional SDEs with multiplicative
noises, \cite{BWY1} for stochastic functional Hamiltonian systems
with degenerate noises, and \cite{BWY2} for functional SPDEs with
additive noises.

\smallskip

In this paper, we also adopt the coupling by change of measure (see,
e.g., the monograph \cite{W13}) to establish the Harnack inequality
for \eqref{eq0}. However, due to the appearance of neutral term, as
we observe below, it becomes more tricky to investigate  the
Harnack inequality for the functional SDE of neutral type
\eqref{eq0}.

\smallskip

For $\xi_x(\theta)\equiv x$ and $\eta_y(\theta)\equiv y$ with
$\theta\in[-r_0,0]$,    from (H1) and (H2), there exists
$\kk_1\in\R$ such that
\begin{equation}\label{*5}
\<Z(x)-Z(y), x-y\> \le -\kk_1|x-y|^2,~~~x,y\in\R^n.
\end{equation}

\begin{lem}\label{Harnack}
{\rm Let the assumptions of Theorem \ref{T1.1} hold. Then, there
exists a constant $c>0$ such that
\begin{equation}\label{q1}
(P_tf(\xi))^2\le
(P_tf^2(\eta))\e^{c\|\xi-\eta\|^2_\8},~~~f\in\B_b(\C),~\xi,~\eta\in
\C,~t>r_0.
\end{equation}
}
\end{lem}

\begin{proof}
We adopt the coupling by change of measures (see, e.g., the
monograph \cite{W13}) to establish the Harnack inequality
\eqref{q1}. Let $\{Y(s)\}_{s\ge0}$ solve an SDE without memory
\begin{equation}\label{*}
\begin{split}
&\d\{Y(s)+LX_s\}\\
&=\Big\{Z(Y(s))+b(X_s)+g(s)1_{[0,\tau)}(s) \cdot
\ff{X(s)-Y(s)}{|X(s)-Y(s)|}\Big\}\d s+\si\d W(s),~~~s>0,
\end{split}
\end{equation}
with $Y_0=\eta\in\C$, where
$$\tau:=\inf\{s\ge 0: X(s)=Y(s)\}$$ is the coupling time and $g:[0,\infty)\mapsto\R_+$  is a continuous mapping to be
determined. By \eqref{*5} and the chain rule,
\begin{equation*}
\begin{split}
\d(\e^{\kk_1s}|X(s)-Y(s)|)
%&=\kk_1\e^{\kk_1s}|X(s)-Y(s)|\d
%s+\e^{\kk_1s}\d|X(s)-Y(s)|\\
&=\kk_1\e^{\kk_1s}|X(s)-Y(s)|\d
s+\e^{\kk_1s}\{-g(s)1_{[0,\tau)}(s)\\
&\quad+|X(s)-Y(s)|^{-1}\<X(s)-Y(s),Z(X(s))-Z(Y(s))\>\}\d s\\
&\le -\e^{\kk_1s}g(s)\d s, ~~~~~s<\tau.
\end{split}
\end{equation*}
Thus, one has
\begin{equation}\label{*4}
\begin{split}
|X(s)-Y(s)|\le
\e^{-\kk_1s}|\xi(0)-\eta(0)|-\e^{-\kk_1s}\int_0^s\e^{\kk_1r}g(r)\d
r,~~~s\le \tau.
\end{split}
\end{equation}In
\eqref{*4}, in particular, choosing \beq\label{CD2}  g(r)=
\ff{|\xi(0)-\eta(0)| \e^{\kk_1r}}{\int_0^t\e^{2\kk_1r}\d r},\ \ \
r\in [0,t]
\end{equation} leads to
\begin{equation}\label{*7}
|X(s) -Y(s)| \le
\ff{|\xi(0)-\eta(0)|(\e^{2\kk_1t-\kk_1s}-\e^{\kk_1s})}{\e^{2\kk_1t}-1},\
\ s\le \tau.
\end{equation} If $t<\tau$, \eqref{*7} implies $X(t)=Y(t)$, which
contradicts the definition of coupling time $\tau.$ Consequently,
we have
  $\tau\le t$. Also,  by the chain
rule, for any $\vv>0$, we obtain from \eqref{*5} and
$X(\tau)=Y(\tau)$ that
\begin{equation*}
\begin{split}
(\vv+|X(t)-Y(t)|^2)^{1/2}&=(\vv+|X(\tau)-Y(\tau)|^2)^{1/2}\\
&\quad+\int_\tau^t(\vv+|X(s)-Y(s)|^2)^{-1/2}\<X(s)-Y(s),Z(X(s))-Z(Y(s))\>\d
s\\
&\le
\ss\vv-\kk_1\int_\tau^t(\vv+|X(s)-Y(s)|^2)^{-1/2}(\vv+|X(s)-Y(s)|^2-\vv)\d
s\\
&\le(1+|\kk_1|(t-\tau))\ss\vv+|\kk_1|\int_\tau^t(\vv+|X(s)-Y(s)|^2)^{1/2}\d
s,~~~t>\tau.
\end{split}
\end{equation*}
Taking $\vv\downarrow0$ and utilizing Gronwall's inequality, we
conclude that $X(t)=Y(t)$ for any $t\ge\tau.$ Hence,
$X_{t+r_0}=Y_{t+r_0}$, and
 \beq\label{CD3} |X(s) -Y(s)| \le
\ff{|\xi(0)-\eta(0)|(\e^{2\kk_1t-\kk_1s}-\e^{\kk_1s})}{\e^{2\kk_1t}-1}=:G(s),\
\ s\le t.\end{equation} Let
\begin{equation*}
\tt W(s):=\int_0^s\si^{-1}h(r)\d r+W(s),~~~s\in[0,t+r_0],
\end{equation*}
where
\begin{equation*}
\begin{split}
h(r):&=\{\xi(r-r_0)-\eta(r-r_0)+Y(r)-X(r)\}1_{[0,r_0]}(r)+\Big(\int_{r-r_0}^r\LL(u)\d
u\Big)1_{(r_0,t+r_0]}(r)\\
&\quad+\Big\{1_{[0,\tau)} g(r)\ff{X(r)-Y(r)}{|X(r)-Y(r)|} +
b(X_r)-b(Y_r)\Big\}\\
&=:h_1(r)1_{[0,r_0]}(r)+h_2(r)1_{(r_0,t+r_0]}(r)+h_3(r).
\end{split}
\end{equation*}
Herein $ \Lambda(u):=Z(Y(u))-Z(X(u))+g(u)1_{[0,\tau)}(u) \cdot
\ff{X(u)-Y(u)}{|X(u)-Y(u)|},u\in[s-r_0,s]. $ For any $s\in[0,r_0],$
note that
\begin{equation}\label{eq2}
\begin{split}
&\int_0^sh_1(r)\d r+\int_0^{r_0}\{\eta(u-r_0)-\xi(u-r_0)\}\d
u\\
%&=-\int_0^s\{\eta(u-r_0)-\xi(u-r_0)\}\d u+\int_0^{r_0}\{\eta(u-r_0)-\xi(u-r_0)\}\d u\\
%&\quad+\int_0^s\{Y(r)-X(r)\}\d r\\
&=\int_{-r_0}^{-s}\{\eta(s+\theta)-\xi(s+\theta)\}\d
\theta+\int_{-s}^0\{Y(s+\theta)-X(s+\theta)\}\d\theta\\
&=\int_{-r_0}^0\{\eta(s+\theta)-\xi(s+\theta)\}1_{\{s+\theta\le0\}}\d\theta+\int_{-r_0}^0\{Y(s+\theta)-X(s+\theta)\}1_{\{s+\theta>0\}}\d\theta\\
&= L(Y_s-X_s).
\end{split}
\end{equation}

\smallskip

From \eqref{eq0} and \eqref{*}, it is trivial to see that
\begin{equation}\label{eq1}
Y(r)-X(r)=\eta(0)-\xi(0)+\int_0^r\LL(u)\d u,~~~~r\ge0.
\end{equation}
For arbitrary $s\in(r_0,t+r_0],$ according to \eqref{eq1}, it
follows that
\begin{equation}\label{eq3}
\begin{split}
&(\eta(0)-\xi(0))r_0+\int_0^{r_0}(r_0-u)\LL(u)\d
u+\int_{r_0}^sh_2(r)\d
r\\
&=(\eta(0)-\xi(0))r_0+\int_0^{r_0}(r_0-u)\LL(u)\d u\\
&\quad+\int_{r_0}^s\Big(r_0\int_0^{r-r_0}\LL(u)\d u +\int_{r-r_0}^r(r-u)\LL(u)\d u\Big)^\prime\d r\\
&=(\eta(0)-\xi(0))r_0+r_0\int_0^{s-r_0}\LL(u)\d u +\int_{s-r_0}^s(s-u)\LL(u)\d u\\
&=\int_{-r_0}^0\Big\{\eta(0)-\xi(0)+\int_0^{s+\theta}\LL(u)\d
u\Big\}\d\theta\\
&= L(Y_s-X_s).
\end{split}
\end{equation}
Hence, from  \eqref{eq2} and \eqref{eq3},  we arrive at
\begin{equation}\label{eq4}
\d L(Y_s-X_s)=\{h_1(s)1_{[0,r_0]}(s)+h_2(s)1_{(r_0,t+r_0]}(s)\}\d
s,~~~~s\in[0,r_0+t].
\end{equation}

\smallskip
Observing that $G(s)$, defined in \eqref{CD3}, is decreasing
for $s\in[0,t]$ and taking \eqref{CD3} into account gives that
\begin{equation}\label{*9}
\|X_s-Y_s\|_\infty^2 \le 1_{[0,r_0]}(s) \|\xi-\eta\|_\infty^2 +
1_{(r_0,t+r_0]}(s) \ff{|\xi(0)-\eta(0)|^2
(\e^{2\kk_1t-\kk_1(s-r_0)}-\e^{\kk_1(s-r_0)})^2}{(\e^{2\kk_1t}-1)^2}.
\end{equation}
By  (H1), \eqref{CD3} and \eqref{*9}, for any $\aa,\bb,\dd>0,$ we
derive from H\"older's inequality  that
 \begin{equation}\label{*8}
 \begin{split}
 &|h(s)|^2\\
 %&=|h_1(s)1_{[0,r_0]}(s)+h_2(s)1_{(r_0,t+r_0]}(s)+h_3(s)|^2\\
 %&\le
 %(1+\aa^{-1})|h_1(s)1_{[0,r_0]}(s)+h_2(s)1_{(r_0,t+r_0]}(s)|^2+(1+\aa)|h_3(s)|^2\\
 %&=(1+\aa^{-1})\{|h_1(s)|^21_{[0,r_0]}(s)+|h_2(s)|^21_{(r_0,t+r_0]}(s)\}+(1+\aa)|h_3(s)|^2\\
 &\le(1+\aa^{-1})\Big\{4\|\xi-\eta\|_\8^21_{[0,r_0]}(s)+r_0\int_{s-r_0}^t\Big((1+\dd)1_{[0,t]} (u)g^2(u)\\
 &\quad+(1+\dd^{-1})L_1^2|X(u)-Y(u)|^2\Big)\d u1_{(r_0,t+r_0]}(s)\\
 &\quad+r_0\int_t^s\Big((1+\dd)1_{[0,t]} (u)g^2(u)\\
 &\quad+(1+\dd^{-1})L_1^2|X(u)-Y(u)|^2\Big)\d u1_{(r_0,t]}(s)\Big\}\\
 &\quad+(1+\aa)\{(1+\bb)g^2(s)1_{[0,t]}(s)+(1+\bb^{-1})L_2^2\|X_s-Y_s\|_\8^2\}\\
 &\le(1+\aa^{-1})\Big\{4\|\xi-\eta\|_\8^21_{[0,r_0]}(s)+r_0\int_{s-r_0}^t\Big((1+\dd)1_{[0,t]} (u)g^2(u)\\
 &\quad+(1+\dd^{-1})L_1^2|X(u)-Y(u)|^2\Big)\d u1_{(r_0,t+r_0]}(s)\Big\}\\
 &\quad+(1+\aa)\{(1+\bb)g^2(s)1_{[0,t]}(s)+(1+\bb^{-1})L_2^2\|X_s-Y_s\|_\8^2\}\\
 &\le(1+\aa^{-1})\Big\{4\|\xi-\eta\|_\8^21_{[0,r_0]}(s)\\
 &\quad+\ff{2r_0\kk_1(1+\dd)|\xi(0)-\eta(0)|^2}{(\e^{2\kk_1t}-1)^2}\Big(\e^{2\kk_1t}-\e^{2\kk_1(s-r_0)}\Big)1_{(r_0,t+r_0]}(s)\\
 &\quad+r_0L_1^2(1+\dd^{-1})(t-s+r_0)G^2(s-r_0)1_{(r_0,t+r_0]}(s)\Big\}\\
 &\quad+(1+\aa)\Big\{\ff{4(1+\bb)\kk_1^2|\xi(0)-\eta(0)|^2\e^{2\kk_1s}}{(\e^{2\kk_1t}-1)^2}1_{[0,t]}(s)\\
 &\quad+(1+\bb^{-1})L_2^2 \|\xi-\eta\|_\infty^21_{[0,r_0]}(s)+
(1+\bb^{-1})L_2^2G^2(s-r_0)1_{(r_0,t+r_0]}(s)\Big\}.
\end{split}
\end{equation}
This, together with a straightforward calculation, yields
\begin{equation*}
\begin{split}
\E\exp\Big(\ff{1}{2}\int_0^{t+r_0}|\si^{-1}h(r)|^2\d r\Big)<\8.
\end{split}
\end{equation*}
As a result,  Novikov's condition holds so that, by the Girsanov
theorem, $\{\tt W(s)\}_{t\in[0,t+r_0]}$ is a Brownian motion under
the weighted probability measure $\d\Q:= R\d\P$ with
 $$R:= \exp\bigg[-\int_0^{t+r_0}\<\si^{-1}h(s), \d W(s)\>-\ff 1 2 \int_0^{t+r_0} |\si^{-1}h(s)|^2 \d s\bigg].$$
Due to \eqref{eq4}, equation \eqref{*} can be reformulated as
\begin{equation*}
\begin{split}
\d\Big\{Y(s)+\int_{-\tau}^0Y(s+\theta)\nu(\d\theta)\Big\}
=\{Z(Y(s))+b(Y_s)\}\d s+\si\d  \tt W(s)
\end{split}
\end{equation*}
with the initial data $Y_0=\eta\in\C.$  By virtue of  weak
uniqueness of solutions and $X_{t+r_0}(\xi)=Y_{t+r_0}(\eta)$, in
addition to the H\"older inequality, one finds,
\begin{equation}\label{*10}
\begin{split}
(P_{t+r_0}f(\eta))^2&=(\E f(X_{t+r_0}(\eta)))^2=(\E_\Q
f(Y_{t+r_0}(\eta)))^2=(\E [Rf(Y_{t+r_0}(\eta))])^2\\
&=(\E [Rf(X_{t+r_0}(\xi))])^2\le  \E R^2  P_{t+r_0}f^2(\xi),
\end{split}
\end{equation}
and that
  \beq\label{NW2}
\beg{split} &\E   R^2\\  &\le \Big(\E \e^{6\int_0^{t+r_0}
|\si^{-1}h(s)|^2\d s }\Big)^{1/2}
\Big(\E\e^{-4\int_0^{t+r_0}\<\si^{-1}h(s), \d B(s)\>-8
\int_0^{t+r_0} |\si^{-1}h(s)|^2 \d
 s}\Big)^{1/2}\\
 & =\Big(\E
\e^{6\int_0^{t+r_0} |h(s)|^2\d s }\Big)^{1/2}.
% &\le \exp\bigg[\ff{p(p+1)(1+\dd)\|\si^{-1}\|^2}{2(p-1)^2}\Big(
%\ff{2\kk_1(1+\|\nu\|)^2|\xi(0)-\eta(0)|^2
%}{\e^{2\kk_1t}-1}\\
 %&\quad +\dd^{-1}(1+\dd)\|\nu\|^2L_1^2|\xi(0)-\eta(0)|^2t +\dd^{-2}(1+\dd)L_2^2\Big(r_0\|\xi-\eta\|_\8^2\\
 %&\quad +\ff{|\xi(0)-\eta(0)|^2(\e^{4\kk_1t}-
 %1-4\kk_1t\e^{4\kk_1t})}{2\kk_1(\e^{2\kk_1t}-1)^2}\Big)\bigg].
 \end{split}\end{equation}
Hence,  from \eqref{*8}-\eqref{NW2},
 for any $t_0>r_0,p>1,\aa,\bb,\dd>0$,
positive $f\in \B_b(\C),$  and $\xi,\eta\in\C$, we deduce that
\begin{equation}\label{*6}
\begin{split}
&(P_{t_0}f(\eta))^2\\&\le P_{t_0}f^2(\xi)
\exp\bigg[3\|\si^{-1}\|^2\Big\{ (1+\aa^{-1})\Big(4\|\xi-\eta\|_\8^2r_0\\
 &\quad+\ff{r_0(1+\dd)|\xi(0)-\eta(0)|^2}{(\e^{2\kk_1t}-1)^2}\Big(2\kk_1t\e^{2\kk_1t}+1-\e^{2\kk_1t}\Big)\\
 &\quad+\ff{r_0L_1^2(1+\dd^{-1})|\xi(0)-\eta(0)|^2}{4\kk_1^2(\e^{2\kk_1t}-1)^2}\times\Big((2\kk_1t-1)(\e^{4\kk_1t}-4\kk_1t\e^{2\kk_1t}-1)\\
 &\quad+2(\e^{2\kk_1t}-1)+4\kk_1t\e^{2\kk_1t}(2\kk_1-1)\Big)\Big)\\
 &\quad+(1+\aa)\Big(\ff{4(1+\bb)\kk_1|\xi(0)-\eta(0)|^2}{2(\e^{2\kk_1t}-1)}\\
 &\quad+(1+\bb^{-1})L_2^2r_0
 \|\xi-\eta\|_\infty^2+\ff{(1+\bb^{-1})L_2^2|\xi(0)-\eta(0)|^2(\e^{4\kk_1t}-4\kk_1t\e^{2\kk_1t}-1)}{2\kk_1(\e^{2\kk_1t}-1)^2}\Big)\Big\}.
\end{split}
\end{equation}
By the Markov property and Schwartz's inequality, for any $t>0$, we
deduce from \eqref{*6}   that there exists $c_0>0$ such that
\begin{equation*}\begin{split}
|P_{t+t_0} f(\xi)|^2& = |\E (P_{t_0} f)(X_t(\xi))|^2  \le \Big(\E \ss{ (P_{t_0} f^2 (X_t(\eta)))\exp[c_0 \|X_t(\xi)-X_t(\eta)\|_\infty^2]}\Big)^2\\
&\le (\E (P_{t_0} f^2 (X_t(\eta)))  \E
\e^{c_0\|X_t(\xi)-X_t(\eta)\|_\infty^2}\\
&= (P_{t+t_0}f^2(\eta)) \E \e^{c_0\|X_t(\xi)-X_t(\eta)\|_\infty^2}.
\end{split}
\end{equation*}
At last, the desired assertion follows from \eqref{w1} immediately.

\end{proof}

\begin{rem}
{\rm In Lemma \ref{Harnack}, we only investigate the Harnack
inequality for a special  class of functional SDEs of neutral type.
It is still open on how to establish the Harnack inequality for
functional SDEs of neutral type with general neutral terms. }
\end{rem}

\noindent{\bf Proof of Theorem \ref{T1.1}} By virtue of Lemma
\ref{general}, Theorem \ref{T1.1} (1)-(3) follows from  Lemma
\ref{exponential}, Lemma \ref{inva}, and Lemma \ref{Harnack}.
  According to \cite[Proposition 2.3]{Wang14},
Theorem \ref{T1.1} (1) implies the desired exponential convergence
of $P_t$ in entropy and in $L^2(\mu)$ in Theorem \ref{T1.1} (3) and
(4), respectively. Moreover, Theorem \ref{T1.1} (5) follows by
carrying out a similar argument to that of \cite[Theorem 1.1
(4)]{BWY}.

\end{document}